\documentclass{article}
\usepackage{amsmath}
\usepackage{amsfonts,amssymb,amsthm,todonotes,enumerate,algpseudocode,algorithm,pgfplots,xassoccnt,fullpage,booktabs,breqn,comment, mathtools, soul,etoolbox}
\usepackage{pgfplotstable}
\pgfplotsset{compat=newest}

\makeatletter
\expandafter\patchcmd\csname\string\algorithmic\endcsname{\itemsep\z@}{\itemsep=3pt}{}{}
\makeatother

\usepackage{float}
\newfloat{algorithm}{t}{lop}

\newtheorem{proposition}{Proposition}

\theoremstyle{definition}

\DeclareCoupledCountersGroup{figtheoexadef}
\DeclareCoupledCounters[name=figtheoexadef]{theorem,lemma,definition,corollary}

\author{Riley Badenbroek \and Etienne de Klerk}
\title{An Analytic Center Cutting Plane Method to Determine Complete Positivity of a Matrix}

\DeclareMathOperator{\diff}{d\!}
\DeclareMathOperator{\interior}{int}

\DeclareMathOperator{\tr}{tr}

\DeclareMathOperator{\dom}{dom}

\DeclareMathOperator{\val}{val}

\DeclareMathOperator{\Diag}{Diag}

\let\vec\relax
\DeclareMathOperator{\vec}{vec}
\DeclareMathOperator{\mat}{mat}

\algnewcommand\Input{\item[\bf Input:] }
\algnewcommand\Output{\item[\bf Output:] }
\newcommand{\conA}{A}
\newcommand{\cona}{a}
\newcommand{\conb}{b}
\newcommand{\conr}{r}
\newcommand{\slackr}{d}
\newcommand{\slackA}{s}
\newcommand{\mulr}{\kappa}
\newcommand{\mulA}{\lambda}

\usepackage{hyperref,breakurl}
\usepackage{cleveref}

\begin{document}

\maketitle

\begin{abstract}
	We propose an analytic center cutting plane method to determine if a matrix is completely positive, and return a cut that separates it from the completely positive cone if not. 
	This was stated as an open (computational) problem by Berman, D\"ur, and Shaked-Monderer [Electronic Journal of Linear Algebra, 2015].
	Our method optimizes over the intersection of a ball and the copositive cone, where membership is determined by solving a mixed-integer linear program suggested by Xia, Vera, and Zuluaga [INFORMS Journal on Computing, 2018]. Thus, our algorithm can, more generally, be used to solve any copositive optimization problem, provided one knows the radius of a ball containing an optimal solution. 
	Numerical experiments show that the number of oracle calls (matrix copositivity checks) for our implementation scales well with the matrix size, growing roughly like $O(d^2)$ for $d\times d$ matrices. 
%	This leads to an effective computational strategy when paired with the recent mixed-integer programming copositivity checks by Badenbroek and De Klerk [arXiv 1907.02368, 2019].
	The method is implemented in Julia, and available at \url{https://github.com/rileybadenbroek/CopositiveAnalyticCenter.jl}.
\end{abstract}

Keywords: copositive optimization, analytic center cutting plane method, completely positive matrices

AMS subject classification: 90C25, 90C51, 49M05, 65K05

\section{Introduction}
\label{sec:Introduction}
We define the completely positive cone $\mathcal{CP}^d \subset \mathbb{S}^d$ as
\begin{equation*}
	\mathcal{CP}^d \coloneqq \{ BB^\top: B \geq 0, B \in \mathbb{R}^{d \times k} \text{ for some $k$} \},
\end{equation*}
where $\mathbb{S}^d$ denotes the space of real symmetric $d \times d$ matrices.
Completely positive matrices play an important role in optimization. For instance, by a theorem of Motzkin and Straus \cite{motzkin1965maxima} (see also De Klerk and Pasechnik \cite{deklerk2002approximation}), the stability number of a graph can be formulated as an optimization problem with linear objective and linear constraints over the completely positive cone (or its dual cone). A seminal result by Burer \cite{burer2009copositive} shows that -- under mild assumptions -- binary quadratic problems can also be reformulated as optimization problems over the completely positive cone. Other applications build on the work by Kemperman and Skibinsky \cite{kemperman1992covariance}, who found that
\begin{equation*}
	\left\{ \int xx^\top \diff \mu(x) : \text{$\mu$ is a finite-valued nonnegative measure supported on $\mathbb{R}^d_+$} \right\} = \mathcal{CP}^d.
\end{equation*}
This equality has spawned a large number of applications in distributionally robust optimization, e.g. Natarajan et al. \cite{natarajan2011mixed} and Kong et al. \cite{kong2013scheduling} (see Li et al. \cite{li2014distributionally} for a survey).

One advantage of these reformulations is that they transform hard problems into linear optimization problems over a proper cone, which allow them to benefit from the (duality) theory of convex optimization. The difficulty in such problems is essentially moved to the conic constraint. It is therefore unsurprising that even testing whether a matrix is completely positive is NP-hard, cf. Dickinson and Gijben \cite{dickinson2014computational}. Several approaches to this testing problem exist in the literature.

Jarre and Schmallowsky \cite{jarre2009computation} propose an augmented primal-dual method that provides a certificate if $C \in \mathcal{CP}^d$ by solving a sequence of second-order cone problems. However, their algorithm converges slowly if $C$ is on the boundary of $\mathcal{CP}^d$, and the regularization they propose to solve this is computationally expensive.

An obvious way to verify that $C$ is completely positive is to find a factorization $C = BB^\top$ where $B \geq 0$. Several authors have done this for specific matrix structures, see Dickinson and D\"ur \cite{dickinson2012linear}, Bomze \cite{bomze2018building}, and the references therein. For general matrices, factorization methods have been proposed by Nie \cite{nie2014Atruncated}, and Sponsel and D\"ur \cite{sponsel2014factorization}, but these methods do not perform well on bigger matrices. Groetzner and D\"ur \cite{groetzner2018factorization} develop an alternating projection scheme that does scale well, but is not guaranteed to find a factorization for a given completely positive matrix. The method struggles in particular for matrices near the boundary of the completely positive cone.
%struggles to find a factorization of completely positive matrices near the boundary of the completely positive cone. 
Another heuristic method based on projection is given by Elser \cite{elser2017matrix}. Sikiri{\'c} et al. \cite{sikiric2020simplex} can find a rational factorization whenever it exists, although the running time is hard to predict.

To actually optimize over the completely positive cone is even harder. Bomze et al. \cite{bomze2011quadratic} suggest a factorization heuristic with promising numerical performance.
A more naive approach to solving completely positive optimization problems is to replace the cone $\mathcal{CP}^d$ with a tractable outer approximation, such as the cone of doubly nonnegative matrices (i.e. the symmetric positive semidefinite matrices with nonnegative elements). If the problem over this outer approximation has an optimal solution $C$, one would not only like to check if $C \in \mathcal{CP}^d$, but also to generate a cut that separates $C$ from $\mathcal{CP}^d$ if $C \notin \mathcal{CP}^d$. After adding the cut to the relaxation, the relaxation may be re-solved, hopefully yielding a better solution (this scheme is mentioned in e.g. Sponsel and D\"ur \cite{sponsel2014factorization} and Berman, D\"ur, and Shaked-Monderer \cite{berman2015open}).

Burer and Dong \cite{burer2013separation} proposed a method to generate such a cut for $5 \times 5$ matrices. Sponsel and D\"ur \cite{sponsel2014factorization} suggested an algorithm based on simplicial partition. 
Nevertheless, finding a cutting plane for the completely positive matrices is still listed as an open problem by Berman, D\"ur, and Shaked-Monderer \cite{berman2015open}.

Our approach will optimize over the dual cone of $\mathcal{CP}^d$ (with respect to the trace inner product $\langle \cdot, \cdot \rangle$), which is known as the copositive cone.
%In this paper, we propose an analytic center cutting plane method to check if a matrix $C$ is completely positive, and if not, to return a deep cut.
%Our method optimizes over the dual cone of $\mathcal{CP}^d$, which is known as the copositive cone. 
This cone is defined as
\begin{equation*}
	\mathcal{COP}^d \coloneqq \{ X \in \mathbb{S}^d : y^\top X y \geq 0 \text{ for all $y \in \mathbb{R}^d_+$} \}.
\end{equation*}
It is well known that $C \in \mathcal{CP}^d$ if and only if $\langle C, X \rangle \geq 0$ for all $X \in \mathcal{COP}^d$. Hence, minimizing $\langle C, X \rangle$ over $X \in \mathcal{COP}^d$ should give us an answer to the question if $C \in \mathcal{CP}^d$ or not, and if not, we immediately have an $X \in \mathcal{COP}^d$ that induces a valid cut.

It should be noted that determining if a matrix $X$ lies in $\mathcal{COP}^d$ is co-NP-complete, see Murty and Kabadi \cite{murty1987some}. The classical copositivity test is due to Gaddum \cite{gaddum1958linear}, but his procedure requires performing a test for all principal minors of a matrix, which does not scale well to larger $d$. Nie et al. \cite{nie2018complete} have proposed an algorithm based on semidefinite programming that terminates in finite time, although the actual computation time is hard to predict. Anstreicher \cite{anstreicher2020testing} shows that copositivity can be tested by solving a mixed-integer linear program (MILP), building on work by Dickinson \cite{dickinson2019new}. See Hiriart-Urruty and Seeger \cite{hiriart2010variational} for a review of the properties of copositive matrices.

Our chosen method of testing if a matrix $X$ is copositive is the same as in Badenbroek and De Klerk \cite{badenbroek2019simulated}, which is similar to Anstreicher's. Our method also solves an MILP, and also admits a $y \geq 0$ such that $y^\top X y < 0$ if $X$ is not copositive. The main difference is that our method derives from Xia et al. \cite{xia2018globally} instead of Dickinson \cite{dickinson2019new}.

Since the copositive cone is intractable, it will have to be replaced by an approximation if we want to optimize over it.
Bundfuss and D\"ur \cite{bundfuss2008algorithmic,bundfuss2009adaptive} use polyhedral inner and outer approximations based on simplicial partitions that are refined in regions interesting to the optimization.
Hierarchies of inner approximations of the copositive cone are proposed by Parrilo \cite{parrilo2000structured}, De Klerk and Pasechnik \cite{deklerk2002approximation} (see also Bomze and De Klerk \cite{bomze2002solving}) and Pe\~na et al. \cite{pena2007computing}. Y{\i}ld{\i}r{\i}m \cite{yildirim2012accuracy} proposes polyhedral outer approximations of the copositive cone, and analyzes the gap to the inner approximations by De Klerk and Pasechnik. Finally, Lasserre \cite{lasserre2014new} proposes a spectrahedral hierarchy of outer approximations of $\mathcal{COP}^d$.

Our approach to optimize over the copositive cone is to use an analytic center cutting plane method. Therefore, it is convenient to use a simple polyhedral outer approximation of the copositive cone: $\{ X \in \mathbb{S}^d : y^\top X y \geq 0 \,\forall y \in \mathcal{Y} \}$, where $\mathcal{Y} \subset \mathbb{R}^d_+$ is a finite set of vectors. These vectors will be generated by performing the copositivity check for some matrix $X$, and if it turns out there exists a $y \geq 0$ such that $y^\top X y < 0$, this $y$ is added to $\mathcal{Y}$.

Analytic center cutting plane methods were first introduced by Goffin and Vial \cite{goffin1993computation} (see \cite{goffin2002convex} for a survey by the same authors, or Boyd et al. \cite{boyd2011accm}).
The advantage of analytic center cutting plane methods is that the number of iterations scales reasonably with the problem dimension. For instance, Goffin et al. \cite{goffin1996complexity} find that the number of iterations is $O^*(n^2/\epsilon^2)$, where $n$ is the number of variables, $\epsilon$ is the desired accuracy, and $O^*$ ignores polylogarithmic terms. 
%Numerical experiments such as in Boyd et al. \cite[Figure 4]{boyd2011accm} suggest that the dependence on $1/\epsilon^2$ in this bound may be overly conservative, and that $\log(1/\epsilon)$
In every iteration of our algorithm, the main computational effort is solving an MILP whose size does not change throughout the algorithm's run.

We describe our method in detail Section \ref{sec:ACCPM} and conduct numerical experiments in Section \ref{sec:NumericalExperiments}.

\subsection*{Notation}
Throughout this work, we use the Euclidean inner product %$\langle \cdot, \cdot \rangle$
on $\mathbb{R}^n$, and the trace inner product $\langle \cdot, \cdot \rangle$ on $\mathbb{S}^d$.
For some vector $x \in \mathbb{R}^n$ with all elements unequal to $0$, and some integer $i \in \mathbb{Z}$, let $x^i \coloneqq \begin{bmatrix} x_1^i & \cdots & x_n^i \end{bmatrix}^\top$.

Since $\mathbb{S}^d$ is isomorphic to $\mathbb{R}^{d(d+1)/2}$, we can also consider our optimization over $\mathcal{COP}^d$ as an optimization over $\mathbb{R}^{d(d+1)/2}$. To do that, we follow the convention from Julia's \texttt{MathOptInterface} package, which (implicitly) uses the following vectorization operator on $X = [X_{ij}] \in \mathbb{S}^d$:
\begin{equation*}
	\vec(X) \coloneqq \begin{bmatrix} X_{11} & X_{12} & X_{22} & X_{13} & X_{23} & \cdots & X_{dd} \end{bmatrix}^\top,
\end{equation*}
i.e. $\vec(X)$ contains the upper triangular part of the matrix.
Let $\mat: \mathbb{R}^{d(d+1)/2} \to \mathbb{S}^d$ be the inverse of $\vec$, and let $\mat^*: \mathbb{S}^d \to \mathbb{R}^{d(d+1)/2}$ be the adjoint of $\mat$.
%with respect to the Euclidean inner product on $\mathbb{R}^{d(d+1)/2}$ and the trace inner product on $\mathbb{S}^d$ 
%(note that $\mat^* \neq \vec$).

The problem we will look at for some fixed $C \in \mathbb{S}^d$ is
\begin{equation}
	\label{eq:CompletelyPositiveProblem}
%	\min_x \{ \langle \mat^*(C), x \rangle : \| x \|^2 \leq 1, \mat(x) \in \mathcal{COP}^d \}.
	\min_X \{ \langle C, X \rangle : \| \vec(X) \|^2 \leq 1, X \in \mathcal{COP}^d \}.
\end{equation}
%If $x \in \mathbb{R}^{d(d+1)/2}$ is an optimal solution to \eqref{eq:CompletelyPositiveProblem}, then $C \in \mathcal{CP}^d$ if and only if $\langle \mat^*(C), x \rangle = \langle C, \mat(x) \rangle \geq 0$.
If $X \in \mathcal{COP}^d$ is an optimal solution to \eqref{eq:CompletelyPositiveProblem}, then $C \in \mathcal{CP}^d$ if and only if $\langle C, X \rangle \geq 0$.

\section{An Analytic Center Cutting Plane Method}
\label{sec:ACCPM}
Analytic center cutting plane methods can be used to solve optimization problems of the form
\begin{equation}
	\label{eq:ACCPProblem}
	\inf_x \{ c^\top x : x \in \mathcal{X} \subseteq \mathbb{R}^n \},
\end{equation}
where $c \in \mathbb{R}^n$ and $\mathcal{X}$ is a nonempty, bounded, convex set for which we know a separation oracle. In other words, given some point $x \in \mathbb{R}^n$, one should be able to determine if $x \in \mathcal{X}$ or not, and moreover, if $x \notin \mathcal{X}$, we must be able to generate a halfspace $\mathcal{H}$ such that $\mathcal{X} \subseteq \mathcal{H}$ but $x \notin \mathcal{H}$.

The idea behind analytic center cutting plane methods is to maintain a tractable outer approximation of the optimal set of \eqref{eq:ACCPProblem}. Then, in every iteration $k$, one approximates the analytic center $x_k$ of this set. One of two things will happen:
\begin{itemize}
	\item $x_k$ does not lie in $\mathcal{X}$. In this case, use a separating hyperplane to remove $x_k$ from the outer approximation of the feasible set.
	\item $x_k$ lies in $\mathcal{X}$. Since $x_k$ is feasible for \eqref{eq:ACCPProblem}, any optimal solution must have an objective value that is at least as good as $\langle c, x_k \rangle$. Any optimal solution will therefore lie in the halfspace $\{x \in \mathbb{R}^n: \langle c, x \rangle \leq \langle c, x_k \rangle \}$, and one may thus restrict the outer approximation to this halfspace.
\end{itemize}

The constraint $\| x \|^2 \leq 1$ from \eqref{eq:CompletelyPositiveProblem} will be included in our outer approximation explicitly. Hence, there are three remaining questions we have to answer before we can solve \eqref{eq:CompletelyPositiveProblem}:
\begin{enumerate}
	\item How will we generate separating hyperplanes for the constraint $\mat(x) \in \mathcal{COP}^d$? %This point is non-trivial, since determining if a matrix is copositive is co-NP-complete, cf. Murty and Kabadi \cite{murty1987some}.
	\item How do we compute the analytic center of the outer approximation?
	\item How can one prune constraints that do not have a large influence on the location of the analytic center?
\end{enumerate}
These three questions will be answered in Sections \ref{subsec:GeneratingCuts}, \ref{subsec:AnalyticCenter}, and \ref{subsec:Pruning}, respectively. Then, we state our algorithm in Section \ref{subsec:Algorithm} and make some remarks concerning its complexity in Section \ref{subsec:Complexity}.

\subsection{Generating Cuts}
\label{subsec:GeneratingCuts}
The first question we will answer is how to generate separating hyperplanes for the copositive cone. Note that $X \in \mathbb{S}^d$ is copositive if and only if
\begin{equation}
	\label{eq:CopositiveMembership}
	\min_y \{ y^\top X y : e^\top y = 1, y \geq 0 \},
\end{equation}
where $e$ is the all-ones vector, is nonnegative. It was shown by Xia, Vera, and Zuluaga \cite{xia2018globally} that the value of \eqref{eq:CopositiveMembership} is equal to the optimal value of the following mixed-integer linear program:
\begin{align}
	\left.
	\begin{aligned}
		\min_{y, z, \mu, \nu} \, & - \mu\\
		\text{subject to }\, & X y + \mu e - \nu = 0\\
		& e^\top y = 1\\
		& 0 \leq y_i \leq z_i && \forall i = 1, ..., d\\
		& 0 \leq \nu_i \leq 2 d (1 - z_i) \max_{k,l} |X_{kl}|  && \forall i = 1, ..., d\\
		& z_i \in \{0,1\} && \forall i = 1, ..., d,
	\end{aligned}
	\label{eq:CopositiveMembershipMILP}
	\right\}
\end{align}
and that any optimal $y$ from \eqref{eq:CopositiveMembershipMILP} is also an optimal solution for \eqref{eq:CopositiveMembership}.
If the optimal value of \eqref{eq:CopositiveMembershipMILP} is nonnegative, then $X$ is copositive. If the optimal value of \eqref{eq:CopositiveMembershipMILP} is negative, then an optimal solution $y \geq 0$ from \eqref{eq:CopositiveMembershipMILP} admits the halfspace $\mathcal{H} = \{ X' \in \mathbb{S}^d : y^\top X' y \geq 0 \}$ such that $\mathcal{COP}^d \subset \mathcal{H}$ but $X \notin \mathcal{H}$. Note that this method was also used in Badenbroek and De Klerk \cite{badenbroek2019simulated}.

As noted in Section \ref{sec:Introduction}, there are alternative methods to test matrix copositivity. Gaddum's method \cite{gaddum1958linear} is already outperformed by the above method for the $6 \times 6$ matrices in our test set, and our MILP method scales considerably better. The method by Nie et al. \cite{nie2018complete} can also become too slow for our purposes at moderate matrix dimensions. Anstreicher's recent method \cite{anstreicher2020testing} also solves an MILP, which we expect to perform similar to \eqref{eq:CopositiveMembershipMILP}.

In theory, we can therefore determine if a matrix is copositive by solving one MILP. In practice however, a solver may return a solution $(\hat{y}, \hat{z}, \hat{\mu}, \hat{\nu})$ to \eqref{eq:CopositiveMembershipMILP} where $\hat{y}^\top \hat{\nu} > 0$, violating the complementarity condition. This is caused by numerical tolerances allowing a solution with $\hat{z} \notin \{ 0,1 \}^d$, which mostly seems to occur if $X$ has low rank (or is close to a low rank matrix). To find the optimal solution if this occurs, we fix $z$ to the element-wise rounded value $\Call{Round}{\hat{z}}$ of $\hat{z}$. If the resulting problem  is still feasible, we can compare its complementary solution with the solution to the model for $z \in \{0,1\}^d \setminus \{\Call{Round}{\hat{z}}\}$. If the constraint $z = \Call{Round}{\hat{z}}$ does make the problem infeasible, we know that any optimal solution will have $z \in \{0,1\}^d \setminus \{\Call{Round}{\hat{z}}\}$. The details of this procedure are given in Algorithm \ref{alg:TestCopositive}, where $\val(\mathcal{M})$ denotes the objective value of the optimal solution returned by the solver when solving the model $\mathcal{M}$.

\begin{algorithm}[!ht]
	\begin{algorithmic}[1]
		\Input Matrix $X \in \mathbb{S}^d$ which we want to test for copositivity.
		
		\Function{TestCopositive}{$X$}
			\State Let $\mathcal{M}$ refer to the model \eqref{eq:CopositiveMembershipMILP} with input $X$
			\State $(\hat{y}, \hat{z}, \hat{\mu}, \hat{\nu}) \gets \Call{SolveModel}{\mathcal{M}}$ \Comment{See Line \ref{line:SolveModel}}
			\If{$\hat{y}^\top X \hat{y} \geq 0$}
				\State \Return true \Comment{Returns true if $X$ is copositive}
			\Else \Comment{Returns a deep cut if $X$ is not copositive:}
				\State \Return $\{ \widehat{X} \in \mathbb{S}^d : \hat{y}^\top \widehat{X} \hat{y} \geq 0 \}$ \Comment{A halfspace $\mathcal{H}$ such that $\mathcal{COP}^d \subseteq \mathcal{H}$ but $X \notin \mathcal{H}$}
			\EndIf
%			\State \Return $\hat{y}$
		\EndFunction
		\Function{SolveModel}{$\mathcal{M}$, $u = +\infty$}\label{line:SolveModel}
			\State Let $(\hat{y}, \hat{z}, \hat{\mu}, \hat{\nu})$ be the solution to the model $\mathcal{M}$ returned by the solver
			\If{$\hat{y}^\top \hat{\nu} > 0$ and $\val(\mathcal{M}) < u$}
				\State Let $\overline{\mathcal{M}}$ be the model $\mathcal{M}$ with the added constraint $z = \Call{Round}{\hat{z}}$
				\State Let $\mathcal{M}'$ be the model $\mathcal{M}$ with the added constraint $\sum_{i: \Call{Round}{\hat{z}_i} = 0} z_i + \sum_{i: \Call{Round}{\hat{z}_i} = 1} (1 - z_i) \geq 1$
				\If{$\overline{\mathcal{M}}$ is feasible}
%					\State Let $(\bar{y}, \bar{z}, \bar{\mu}, \bar{\nu})$ be the solution to the model $\overline{\mathcal{M}}$ returned by the solver
					\State Compute the optimal solution to $\overline{\mathcal{M}}$, and $\Call{SolveModel}{\mathcal{M}', \val(\overline{\mathcal{M}})}$ if $\mathcal{M}'$ is feasible
					\State \Return the solution with the best objective value out of these two
				\Else
					\State \Return $\Call{SolveModel}{\mathcal{M}'}$
				\EndIf
			\Else
				\State \Return the solution $(\hat{y}, \hat{z}, \hat{\mu}, \hat{\nu})$
			\EndIf
		\EndFunction
	\end{algorithmic}
	\caption{Method for testing copositivity or finding deep cuts}
	\label{alg:TestCopositive}
\end{algorithm}

\subsection{Approximating the Analytic Center}
\label{subsec:AnalyticCenter}
Now that we saw how to generate cuts for the copositive cone, we turn our attention to the second question: how to approximate the analytic center of our outer approximation.
For the sake of concreteness, let us suppose the convex body $\mathcal{Q}$ for which we want to approximate the analytic center is the intersection of a ball and a polyhedron, i.e.
\begin{equation}
	\mathcal{Q} = \left\{ x \in \mathbb{R}^n : \| x \|^2 \leq \conr^2, \enspace \cona_i^\top x \leq \conb_i \,\forall i = 1, ..., m \right\},
	\label{eq:FormConvexBody}
\end{equation}
where $\cona_1^\top, ..., \cona_m^\top$ are the rows of a matrix $\conA$, and $\conb \coloneqq (\conb_1, ..., \conb_m)$.
The analytic center of $\mathcal{Q}$ is the optimal solution $x$ to the problem
\begin{equation}
	\label{eq:ACFeasibleForm}
	\inf_{x} \left\{ -\log(\conr^2 - \| x \|^2 ) - \sum_{i=1}^{m} \log(\conb_i- \cona_i^\top x) \right\}.
\end{equation}
It is well known that self-concordant barrier functions only have an analytic center when their domain is bounded (see e.g. Renegar \cite[Corollary 2.3.6]{renegar2001mathematical}). This is why we use the upper bound $\conr$ on $\| x \|$. Of course, a more traditional solution would be to ensure that the linear constraints $\conA x \leq \conb$ describe a bounded set. We decided against this for reasons of numerical stability (more details in Section \ref{subsec:Complexity}).
%Note that $x_*$ is the optimal solution to \eqref{eq:ACFeasibleForm} if and only if
%\begin{equation*}
%	\frac{2}{\conr^2 - \| x \|^2} x + \sum_{i=1}^{m} \frac{1}{\conb_i- \cona_i^\top x} \cona_i = 0.
%\end{equation*}

Since the objective function in \eqref{eq:ACFeasibleForm} can only be evaluated at $x$ where $\| x \|^2 < \conr^2$ and $\conA x < \conb$, we will use an infeasible-start Newton method to solve \eqref{eq:ACFeasibleForm}.
Similar to Boyd et al. \cite[Section 2]{boyd2011accm}, one can reformulate the problem of computing the analytic center of $\mathcal{Q}$ as
\begin{equation}
	\label{eq:ACInfeasibleForm}
	\inf_{x,\slackr,\slackA} \left\{ -\log(\slackr) - \sum_{i=1}^{m} \log(\slackA_i): \slackr \leq \conr^2 - \| x \|^2, \slackA \leq \conb - \conA x \right\},
\end{equation}
which has Lagrangian
\begin{equation*}
	L(x,\slackr, \slackA, \mulr, \mulA) = -\log(\slackr) - \sum_{i=1}^{m} \log(\slackA_i) + \mulr( \slackr - \conr^2 + \| x \|^2 ) + \mulA^\top (\slackA - \conb + \conA x),
\end{equation*}
with gradient
\begin{equation}
	\nabla L(x,\slackr, \slackA, \mulr, \mulA) = \begin{bmatrix}
		2 \mulr x + \conA^\top \mulA\\
		-\slackr^{-1} + \mulr\\
		-\slackA^{-1} + \mulA\\
		\slackr - \conr^2 + \| x \|^2\\ 
		\slackA - \conb + \conA x
	\end{bmatrix}.
	\label{eq:LagrangianGradient}
\end{equation}
For the sake of completeness, let us show that it suffices to compute a stationary point of the Lagrangian.
\begin{proposition}
	Let $\conA \in \mathbb{R}^{m \times n}$ have rows $\cona_1^\top, ..., \cona_m^\top$, and let $\conb \in \mathbb{R}^m$, and $\conr > 0$.
	Let $\mathcal{Q}$ be as defined in \eqref{eq:FormConvexBody}, and assume that it has nonempty interior. Then, $x_*$ is the analytic center of $\mathcal{Q}$ if and only if there exist $\slackr_*, \slackA_*, \mulr_*, \mulA_* > 0$ such that $\nabla L(x_*,\slackr_*, \slackA_*, \mulr_*, \mulA_*) = 0$.
\end{proposition}
\begin{proof}
	Because $\mathcal{Q}$ is nonempty and bounded, it has an analytic center, and problem \eqref{eq:ACInfeasibleForm} has an optimal solution.
	Since \eqref{eq:ACInfeasibleForm} is convex, a feasible solution $(x_*,\slackr_*,\slackA_*)$ is optimal if and only if it satisfies the KKT conditions: there should exist $\mulr_*, \mulA_*$ such that
	\begin{equation*}
		\begin{cases}
%			\nabla_{x,d,s} L(x_*,\slackr_*, \slackA_*, \mulr_*, \mulA_*) = 
			\begin{bmatrix}
				2 \mulr_* x_* + \conA^\top \mulA_*\\
				-\slackr_*^{-1} + \mulr_*\\
				-\slackA_*^{-1} + \mulA_*
			\end{bmatrix} = 0\\ 
			\mulr_*( \slackr_* - \conr^2 + \| x_* \|^2 ) = 0\\ 
			\mulA_*^\top (\slackA_* - \conb + \conA x_*) = 0\\
			\slackr_* \leq \conr^2 - \| x_* \|^2\\ 
			\slackA_* \leq \conb - \conA x_*\\
			\mulr_*, \mulA_* \geq 0.
		\end{cases}
	\end{equation*}
	Since $\mulr_* = \slackr_*^{-1} > 0$ and $\mulA_* = \slackA_*^{-1} > 0$, the claim follows.
\end{proof}

%is feasible, and there exist $\mulr, \mulA \geq 0$ for which $\mulr( \slackr - \conr^2 + \| x \|^2 ) = 0$, $\mulA^\top (\slackA - \conb + \conA x) = 0$, and $\nabla_x L(x,\slackr, \slackA, \mulr, \mulA) = -2 \mulr x - \conA^\top \mulA = 0$.

%As argued in Boyd and Vandenberghe \cite[Section 10.3.1]{boyd2004convex}, the optimality conditions for problem \eqref{eq:ACInfeasibleForm} are given by Lagrangian stationarity, i.e. a point $(x,\slackr, \slackA, \mulr, \mulA)$ is optimal only if $\nabla L(x,\slackr, \slackA, \mulr, \mulA) = 0$.
The first order approximation for the Lagrangian shows
\begin{equation}
	\nabla L(x+\Delta x,\slackr + \Delta\slackr, \slackA+\Delta\slackA, \mulr+\Delta\mulr, \mulA+\Delta\mulA)
	\approx \nabla L(x,\slackr, \slackA, \mulr, \mulA) + \nabla^2 L(x,\slackr, \slackA, \mulr, \mulA) \begin{bmatrix} \Delta x\\ \Delta \slackr\\ \Delta \slackA\\ \Delta \mulr\\ \Delta \mulA \end{bmatrix}, 
	\label{eq:LagrangianFirstOrderApprox}
\end{equation}
which means we can solve a linear system to find the Newton step $(\Delta x, \Delta \slackr, \Delta \slackA, \Delta \mulr, \Delta \mulA)$ with which we can approximate a stationary point of $L$.
%makes $\nabla L(x+\Delta x,\slackr + \Delta\slackr, \slackA+\Delta\slackA, \mulr+\Delta\mulr, \mulA+\Delta\mulA) \approx 0$.
Next, we find that
\begin{equation}
	\nabla^2 L(x,\slackr, \slackA, \mulr, \mulA) = \begin{bmatrix}
		2 \mulr I & 0 & 0 & 2 x & \conA^\top\\
		0 & \slackr^{-2} & 0 & 1 & 0\\
		0 & 0 & \Diag(\slackA^{-2}) & 0 & I\\
		2 x^\top & 1 & 0 & 0 & 0\\
		\conA & 0 & I & 0 & 0
	\end{bmatrix}.
	\label{eq:LagragianHessian}
\end{equation}

Thus, if we substitute the expressions \eqref{eq:LagrangianGradient} and \eqref{eq:LagragianHessian} in \eqref{eq:LagrangianFirstOrderApprox}, we see that the Newton step should satisfy
\begin{equation}
	\label{eq:NewtonSystem}
	0 = \begin{bmatrix}
		2 \mulr x + \conA^\top \mulA\\
		-\slackr^{-1} + \mulr\\
		-\slackA^{-1} + \mulA\\
		\slackr - \conr^2 + \| x \|^2\\ 
		\slackA - \conb + \conA x
	\end{bmatrix} + 
	\begin{bmatrix}
		2 \mulr I & 0 & 0 & 2 x & \conA^\top\\
		0 & \slackr^{-2} & 0 & 1 & 0\\
		0 & 0 & \Diag(\slackA^{-2}) & 0 & I\\
		2 x^\top & 1 & 0 & 0 & 0\\
		\conA & 0 & I & 0 & 0
	\end{bmatrix}
	\begin{bmatrix} \Delta x\\ \Delta \slackr\\ \Delta \slackA\\ \Delta \mulr\\ \Delta \mulA \end{bmatrix}.
\end{equation}
We could solve this system directly, but it is more efficient to note that the last conditions imply
\begin{equation}
	\begin{cases}
		\Delta \mulr &= - \mulr + \slackr^{-1} - \slackr^{-2} \Delta \slackr\\
		\Delta \mulA &= - \mulA  + \slackA^{-1} - \Diag(\slackA^{-2}) \Delta \slackA\\
		\Delta \slackr &= -\slackr + \conr^2 - \| x \|^2 - 2 x^\top \Delta x\\
		\Delta \slackA &= -\slackA + \conb - \conA x - \conA \Delta x
	\end{cases}
	\label{eq:NewtonConditions}
\end{equation}
which means the entire Newton step can be expressed in terms of $\Delta x$. The first $n$ equations of the Newton system \eqref{eq:NewtonSystem} are thus
\begin{align*}
	-2 \mulr x - \conA^\top \mulA
	&= 2 \mulr \Delta x + 2 x \Delta \mulr + \conA^\top \Delta \mulA\\
	&= 2 \mulr \Delta x + 2 x [\slackr^{-1} - \mulr - \slackr^{-2} \Delta \slackr] + \conA^\top [\slackA^{-1} - \mulA - \Diag(\slackA^{-2}) \Delta \slackA]\\
	&= 2 \mulr \Delta x + 2 x [\slackr^{-1} - \mulr - \slackr^{-2} (-\slackr + \conr^2 - \| x \|^2 - 2 x^\top \Delta x)] \\
	&\qquad + \conA^\top [\slackA^{-1} - \mulA - \Diag(\slackA^{-2}) (-\slackA + \conb - \conA x - \conA \Delta x)]
%	&-(2 \lambda^r x + P^\top \lambda^P)\\
%	&= 2\lambda^r \Delta x + 2 \Delta\lambda^r x + P^\top \Delta \lambda^P\\
%	&=2 \lambda^r \Delta x + 2 \left[-(1/s^r + \lambda^r) + \Delta s^r /(s^r)^2\right] x \\
%	&\quad + P^\top \left[ -((1/s^P_1, ..., 1/s^P_m) + \lambda^P) + \Diag(1/(s_1^P)^2, ..., 1/(s_m^P)^2) \Delta s^P\right]\\
%	&= 2\lambda^r \Delta x + 2 \left[-(1/s^r + \lambda^r) + \left[-(s^r - r^2 + \| x \|^2) - 2 x^\top \Delta x\right] /(s^r)^2\right] x \\
%	&\quad + P^\top \left[ -((1/s^P_1, ..., 1/s^P_m) + \lambda^P) + \Diag(1/(s_1^P)^2, ..., 1/(s_m^P)^2) \left[ -(s^P - q + P x) - P \Delta x \right]\right].
\end{align*}
or equivalently,
\begin{equation}
	\label{eq:NewtonSystemDeltaX}
%	\left. \begin{aligned}
	\left[ 2 \mulr I + \frac{4}{\slackr^2} x x^\top + \conA^\top \Diag(\slackA^{-2}) \conA \right] \Delta x
	= \frac{\conr^2 - \| x \|^2 - 2\slackr}{\slackr^2} 2 x
	+ \conA^\top \Diag(\slackA^{-2}) (\conb- \conA x -2 \slackA).
%	\end{aligned} \right\}
\end{equation}
After solving this system for $\Delta x$, we can compute the other components of the Newton step through equations \eqref{eq:NewtonConditions}. 
Now that it is clear how one can compute the Newton step for problem \eqref{eq:ACInfeasibleForm}, we propose Algorithm \ref{alg:InfeasibleNewton} to solve \eqref{eq:ACInfeasibleForm}.

\begin{algorithm}[!ht]
	\begin{algorithmic}[1]
		\Input Convex body $\mathcal{Q} = \{ x \in \mathbb{R}^n : \| x \|^2 \leq \conr^2, \conA x \leq \conb \}$, where $\conA \in \mathbb{R}^{m \times n}$ and $\conb \in \mathbb{R}^m$; starting point $x_0 \in \mathbb{R}^n$; maximum number of iterations $k_{\text{max}} = 50$; gradient norm tolerance $\delta = 10^{-8}$.
		
		\Function{AnalyticCenter}{$\mathcal{Q}, x_0$}
		\State $k \gets 1$
		\State $\slackr_0 \gets \begin{cases}
			\conr^2 - \| x_0 \|^2 & \text{if } \conr^2 - \| x \|^2 > 0\\ 1 & \text{otherwise}
		\end{cases}$
		\State $(\slackA_0)_i \gets \begin{cases}
			\conb_i - \cona_i^\top x & \text{if } \conb_i - \cona_i^\top x > 0\\ 1 & \text{otherwise}
		\end{cases}$ \quad for all $i = 1, ..., m$
		\State $\mulr_0 \gets -1$
		\State $\mulA_0 \gets 0$
		\While{$k \leq k_{\text{max}}$}
			\State Compute $\Delta x_k$ from \eqref{eq:NewtonSystemDeltaX}
			\State Compute $(\Delta \slackr_k, \Delta \slackA_k, \Delta \mulr_k, \Delta \mulA_k)$ from \eqref{eq:NewtonConditions}
			\State $t_k \gets \min\{ 1, 0.9 \times \sup\{ t \geq 0 : \slackr_k + t \Delta \slackr_k \geq 0, \slackA_k + t \Delta \slackA_k \geq 0, \mulr_k + t \Delta \mulr_k \geq 0 \}  \}$ \label{line:ComputeT}
			\State $g_k(t) := \| \nabla L(x_k+t\Delta x_k,\slackr_k + t\Delta\slackr_k, \slackA_k+t\Delta\slackA_k, \mulr_k+t\Delta\mulr_k, \mulA_k+t\Delta\mulA_k) \|$
			\If{$g_k(0) \leq \delta$ and $\mulA_k \geq 0$ and ($g_k(t_k) \geq g_k(0)$ or $k = k_{\text{max}}$)}
				\State \Return $x_k$ with success status
			\EndIf
			\State $k \gets k+1$
		\EndWhile
		\State \Return $x_k$ with failure status
		\EndFunction
	\end{algorithmic}
	\caption{Infeasible start Newton method for \eqref{eq:ACInfeasibleForm}}
	\label{alg:InfeasibleNewton}
\end{algorithm}

Let us make a few observations about this algorithm.
First, note that if $\mulr > 0$, the matrix $ 2 \mulr I + 4\slackr^{-2} x x^\top + \conA^\top \Diag(\slackA^{-2}) \conA$ is positive definite, and hence invertible. Thus, as long as $\mulr > 0$, the system \eqref{eq:NewtonSystemDeltaX} will have a (unique) solution $\Delta x$.

Second, the value for $t$ in Line \ref{line:ComputeT} of Algorithm \ref{alg:InfeasibleNewton} is chosen such that after the update, $\slackr$, $\slackA$, and $\mulr$ will all remain positive. In principle, the value $0.9$ could be replaced by any real number from $(0,1)$. Note that we are not requiring that $\mulA$ remains positive in all iterations: numerical evidence suggests that the method is more likely to succeed if some elements of $\mulA$ are allowed to be negative in some iterations. Nevertheless, Algorithm \ref{alg:InfeasibleNewton} only returns a success status if the final $\mulA$ is nonnegative.

Third, the algorithm returns the current solution $x$ with success status in two cases. In either case, the current solution should approximately be a stationary point of the Lagrangian, i.e. the norm of $\nabla L(x,\slackr, \slackA, \mulr, \mulA)$ has to be small, and we should have $\mulA \geq 0$. Moreover, one of the following conditions should hold:
\begin{enumerate}
	\item Updating the point by adding $t$ times the Newton step leads to a larger norm of the Lagrangian gradient. In this case, taking the step does not improve the solution. Since the current point is already approximately a stationary point, this solution is returned;
	\item We are in iteration $k_{\text{max}}$. Since the current point is approximately a stationary point, this solution is returned.
\end{enumerate}
The reason to continue taking Newton steps even if the norm of the Lagrangian's gradient is small is that Newton's method converges very rapidly when the current point is near the optimum. By running just a few more iterations, we get a solution with much higher accuracy.

Finally, compared to the algorithm in Boyd et al. \cite[Section 2]{boyd2011accm}, Algorithm \ref{alg:InfeasibleNewton} does not use backtracking line search. The reason is that for problem \eqref{eq:ACInfeasibleForm}, the norm of the Lagrangian gradient does not seem to decrease monotonically during the algorithm's run. In fact, the norm of this gradient usually first decreases to the order $10^0$, then increases slightly to the order $10^1$, before decreasing rapidly to the order $10^{-8}$. If one does backtracking line search on $t$ to ensure that in every iteration the norm of the gradient decreases, the values of $t$ can become very small (say, of the order $10^{-9}$). Then, the number of iterations required to achieve convergence would be impractically large.

\subsection{Pruning Constraints}
\label{subsec:Pruning}
The next question we should answer is how we can prune constraints from our outer approximation \eqref{eq:FormConvexBody}. Pruning is often used to reduce the number of constraints defining the outer approximation, which means keeps the computational effort per iteration stable. Moreover, the linear system \eqref{eq:NewtonSystemDeltaX} will quickly become ill-conditioned if no constraints are dropped.

The idea we use is the same as in Boyd, Vandenberghe, and Skaf \cite[Section 3]{boyd2011accm}: denote the barrier of which we compute the analytic center by
\begin{equation}
	\label{eq:AnalyticCenterDefBarrier}
	\Phi(x) \coloneqq - \log(\conr^2 - \| x \|^2) - \sum_{i=1}^m \log(\conb_i - \cona_i^\top x).
\end{equation}
Since $\Phi$ is self-concordant, the Dikin ellipsoid around the analytic center of $\Phi$ is contained in $\mathcal{Q}$, i.e.
\begin{equation}
	\label{eq:DikinInnerApproximation}
	\{ x \in \mathbb{R}^n : (x - x_*)^\top \nabla^2 \Phi(x_*) (x-x_*) \leq 1 \} \subseteq
	\mathcal{Q},
\end{equation}
where $x_*$ is the minimizer of $\Phi$ and 
\begin{equation*}
%	\label{eq:AnalyticCenterHessianBarrier}
	\nabla^2 \Phi(x_*) = \frac{2}{\conr^2 - \| x_* \|^2} I + \frac{4 }{(\conr^2 - \| x_* \|^2)^2} x_* x_*^\top + \sum_{i=1}^m \frac{1}{(\conb_i - \cona_i^\top x_*)^2} \cona_i \cona_i^\top.
\end{equation*}
Moreover, it will be shown at the end of this section that for our outer approximation $\mathcal{Q}$ it holds that
\begin{equation}
	\label{eq:DikinOuterApproximation}
	\mathcal{Q} \subseteq \{ x \in \mathbb{R}^n : (x -x_*)^\top \nabla^2 \Phi(x_*) (x-x_*) \leq (m+1)^2 \}.
\end{equation}
Hence, following \cite{boyd2011accm}, we define the relevance measure
\begin{equation}
	\label{eq:DefRelevanceMeasure}
	\eta_i \coloneqq \frac{\conb_i - \cona_i^\top x_*}{\sqrt{\cona_i^\top \nabla^2 \Phi(x_*)^{-1} \cona_i}},
\end{equation}
for all linear constraints $i = 1, ..., m$. By \eqref{eq:DikinInnerApproximation}, all $\eta_i$ are at least one. Moreover, it follows from \eqref{eq:DikinOuterApproximation} that if $\eta_i \geq m+1$, the corresponding constraint is certainly redundant.

With this in mind, we propose Algorithm \ref{alg:Pruning} to prune constraints from $\mathcal{Q}$. Note that the ball constraint $\| x \|^2 \leq r^2$ is never pruned.

\begin{algorithm}[!ht]
	\begin{algorithmic}[1]
		\Input Convex body $\mathcal{Q} = \{ x \in \mathbb{R}^n : \| x \|^2 \leq \conr^2, \conA x \leq \conb \}$, where $\conA \in \mathbb{R}^{m \times n}$ and $\conb \in \mathbb{R}^m$; analytic center $x_*$ of $\mathcal{Q}$; maximum number of linear inequalities $m_{\text{max}} = 3n$.
		
		\Function{Prune}{$\mathcal{Q}, x_*$}
			\If{$m > n$}
				\State Compute $\eta_i$ as in \eqref{eq:DefRelevanceMeasure} for $i = 1, ..., m$
				\State Remove all constraints $\cona_i^\top x \leq \conb_i$ with $\eta_i \geq m+1$ from $\mathcal{Q}$
				\If{$\mathcal{Q}$ still contains more than $m_{\text{max}}$ linear inequalities}
					\State Remove the constraints $\cona_i^\top x \leq \conb_i$ with the largest values of $\eta_i$ from $\mathcal{Q}$ such that $m_{\text{max}}$ remain
				\EndIf
%				\State Remove the $m - m_{\text{max}}$ constraints $\cona_i^\top x \leq \conb_i$ with the largest values of $\eta_i$ from $\mathcal{Q}$
			\EndIf
			\State \Return $\mathcal{Q}$
		\EndFunction
	\end{algorithmic}
	\caption{A pruning method for the intersection of a ball and a polyhedron}
	\label{alg:Pruning}
\end{algorithm}
As an alternative, one might consider dropping $m - m_{\text{max}}$ constraints, possibly keeping some redundant constraints. The reason we do not adopt this strategy is that we noticed Algorithm \ref{alg:Pruning} leads to slightly better numerical performance on our test sets.

We finish this section with a proof of the relation \eqref{eq:DikinOuterApproximation}.

\begin{proposition}
	\label{prop:AnalyticCenterParameter}
	Let $\conA \in \mathbb{R}^{m \times n}$ have rows $\cona_1^\top, ..., \cona_m^\top$, and let $\conb \in \mathbb{R}^m$, and $\conr > 0$. Let $\mathcal{Q}$ be as defined in \eqref{eq:FormConvexBody}, and assume that it has nonempty interior. Define $\Phi$ as in \eqref{eq:AnalyticCenterDefBarrier}, and let $x_*$ be the minimizer of $\Phi$. Then, for any $x \in \dom \Phi$, we have
	\begin{equation*}
		(x-x_*)^\top \nabla^2 \Phi(x_*) (x-x_*) \leq (m+1)^2.
	\end{equation*}
\end{proposition}
\begin{proof}
	Define the barrier function
	\begin{equation*}
		f(t, \tilde{x}, s) \coloneqq - \log(t^2 - \| \tilde{x} \|^2) - \sum_{i=1}^m \log(\slackA_i),
	\end{equation*}
	whose domain is a symmetric cone. The barrier parameter $\vartheta$ of $f$ satisfies $\vartheta \leq m+1$. Note that any $x \in \dom \Phi$ if and only if $(\conr,x, \conb-\conA x) \in \dom f$. We will first show that the gradient of $f$ at $(\conr,x_*, \conb-\conA x_*)$ is orthogonal to $ (\conr,x, \conb-\conA x) - (\conr,x_*, \conb-\conA x_*) = (0, x - x_*, \conA (x_*-x ) )$. The claim will then follow from a property of symmetric cones.
	
	The gradient of $f$ is
	\begin{equation*}
		\nabla f(t,\tilde{x},\slackA) \coloneqq \begin{bmatrix}
			-2 t / (t^2 - \| \tilde{x} \|^2)\\
			2 x / (t^2 - \| \tilde{x} \|^2)\\
			-\slackA^{-1}
		\end{bmatrix},
	\end{equation*}
	so it follows that
	\begin{equation}
		\label{eq:GradientHyperplaneInnerProduct}
		\nabla f(\conr, x_*, \conb - \conA x_*)^\top
		\begin{bmatrix}
			0 \\ x - x_* \\ \conA (x_*-x )
		\end{bmatrix}
		= \frac{2 x_*^\top (x - x_*)}{\conr^2 - \|x_* \|^2} - \sum_{i=1}^m \frac{ \cona_i^\top (x_*-x) }{\conb_i - \cona_i^\top x_*}.
	\end{equation}
	Because $x_*$ is the minimizer of the convex function $\Phi$, we have
	\begin{equation*}
		0 = \nabla\Phi(x_*) = \frac{2}{\conr^2 - \| x_* \|^2} x_* + \sum_{i=1}^m \frac{1}{\conb_i - \cona_i^\top x_*} \cona_i,
	\end{equation*}
	which implies that \eqref{eq:GradientHyperplaneInnerProduct} is zero.
%	Because \eqref{eq:GradientHyperplaneInnerProduct} is zero, 
	Therefore, Theorem 3.5.9 in Renegar \cite{renegar2001mathematical} shows that
	\begin{equation}
		\label{eq:AnalyticCenterParameterIntermediate}
		\begin{bmatrix}
			0 \\ x - x_* \\ \conA (x_*-x )
		\end{bmatrix}^\top \nabla^2 f(\conr, x_*, \conb - \conA x_*) 
		\begin{bmatrix}
			0 \\ x - x_* \\ \conA (x_*-x )
		\end{bmatrix} \leq \vartheta^2,
	\end{equation}
	where $\nabla^2 f$ is the Hessian of $f$, given by
	\begin{equation*}
		\nabla^2 f(t,\tilde{x},s) \coloneqq \frac{1}{(t^2 - \| \tilde{x} \|^2)^2} \begin{bmatrix}
			2t^2 + \| \tilde{x} \|^2  & -4 t \tilde{x}^\top & 0\\
			-4 t \tilde{x}^\top  & 2(t^2 - \| \tilde{x} \|^2) I + 4 \tilde{x} \tilde{x}^\top & 0 \\
			0 & 0 & (t^2 - \| \tilde{x} \|^2)^2 \Diag(s^{-2})
		\end{bmatrix}.
	\end{equation*}
	In other words, \eqref{eq:AnalyticCenterParameterIntermediate} is equivalent to
	\begin{equation*}
		( x - x_*)^\top \left[ \frac{2}{\conr^2 - \| x_* \|^2} I + \frac{4 }{(\conr^2 - \| x_* \|^2)^2} x_* x_*^\top \right] (x - x_*) + \sum_{i=1}^m \frac{(\cona_i^\top (x_* - x) )^2}{(\conb_i - \cona_i^\top x_*)^2} \leq \vartheta^2,
	\end{equation*}
	which proves the claim, since $\vartheta \leq m+1$.
\end{proof}

\subsection{Algorithm Description}
\label{subsec:Algorithm}
Now that we answered the major questions surrounding an ACCP method for checking complete positivity of a matrix, we more on to our final method. We start with a quite general analytic center cutting plane method, and then add a wrapper function that performs the complete positivity check. The reason for making this split is that it makes our code easy to extend when solving other copositive optimization problems for which a bound on the norm of an optimal solution is known. We state our proposed analytic center cutting plane method to solve \eqref{eq:ACCPProblem} in Algorithm \ref{alg:ACCP}.

\begin{algorithm}[!ht]
	\begin{algorithmic}[1]
		\Input Objective $c \in \mathbb{R}^n$; oracle function $\Call{Oracle}{}: \mathbb{R}^n \to \{ \text{true} \} \cup \{ \{ x \in \mathbb{R}^n : a^\top x \leq b \} : a \in \mathbb{R}^n, b \in \mathbb{R} \}$; radius $\conr > 0$; optimality tolerance $\epsilon = 10^{-6}$.
		\Function{ACCP}{$c$, \textsc{Oracle}, $r$}
			\State $\mathcal{Q}_1 \gets \{ x \in \mathbb{R}^n : \| x \|^2 \leq \conr^2 \}$
			\State $x_0 \gets 0$
			\State $k \gets 1$
			\While{the best feasible solution so far $x_*$ has $\Call{RelativeGap}{c, x_*, \mathcal{Q}_k} > \epsilon$} \Comment{See Line \ref{line:RelativeGap}}
				\State $x_k \gets \Call{AnalyticCenter}{\mathcal{Q}_k, x_{k-1}}$
				\If{\Call{AnalyticCenter}{} terminated with a failure status}
					\State Check if $x_k \in \interior \mathcal{Q}_k$. If not, throw an error.
				\Else
					\State $\mathcal{Q}_k \gets \Call{Prune}{\mathcal{Q}_k, x_k}$
				\EndIf
				\If{\Call{Oracle}{$x_k$} returns true}
					\State $\mathcal{Q}_{k+1} \gets \mathcal{Q}_k \cap \{ x \in \mathbb{R}^n : c^\top x / \| c \| \leq c^\top x_k / \| c \| \}$
				\Else \Comment{\Call{Oracle}{$x_k$} returns a halfspace}
					\State $\mathcal{H}_k = \{x \in \mathbb{R}^n : a^\top_k x \leq b_k\}$ is the halfspace returned by \Call{Oracle}{$x_k$}
					\State $\mathcal{Q}_{k+1} \gets \mathcal{Q}_k \cap \{x \in \mathbb{R}^n : a_k^\top x / \| a_k \| \leq b_k / \| a_k \|\}$
				\EndIf
				\State $k \gets k+1$
			\EndWhile
			\State \Return the best feasible solution found $x_*$
		\EndFunction
		\Function{RelativeGap}{$c, x_*, \mathcal{Q}$} \label{line:RelativeGap}
			\State $l \gets \min_x\{ c^\top x : x \in \mathcal{Q} \}$
			\State \Return $(c^\top x_* - l) / (1 + \min\{ |c^\top x_* |, |l| \})$
		\EndFunction
	\end{algorithmic}
	\caption{Analytic Center Cutting Plane method to solve \eqref{eq:ACCPProblem}}
	\label{alg:ACCP}
\end{algorithm}

We continue the algorithm even if we cannot find the analytic center to high accuracy. Late in the algorithm's run, the system \eqref{eq:NewtonSystemDeltaX} often becomes ill-conditioned. This is to be expected, since as Algorithm \ref{alg:ACCP} progresses, the outer approximation $\mathcal{Q}_k$ becomes smaller and smaller. The distance from the analytic center to the linear constraints also goes to zero, but not at the same pace for every constraint. We may arrive in a situation where $\conb_i - \cona_i^\top x_k$ is of the order $10^{-4}$ for some constraints $i$, and of the order $10^{-8}$ for other constraints. This causes a considerable spread in the eigenvalues of the matrix in \eqref{eq:NewtonSystemDeltaX}.

If the analytic center is not known to a decent accuracy, the pruning procedure in Algorithm \ref{alg:Pruning} may remove constraints that are actually very important to the definition of $\mathcal{Q}_k$. One could of course still run the pruning function using the inaccurate analytic center approximation. However, because the problems in the analytic center computation only occur late in the algorithm's run, pruning or not pruning with the inaccurate approximation does not seem to have a major impact on total runtime.

Algorithm \ref{alg:ACCP} is a (relatively) general analytic center cutting plane method. The problem \eqref{eq:CompletelyPositiveProblem} can be solved by calling Algorithm \ref{alg:ACCP} with the right parameters, as is done by Algorithm \ref{alg:CompletelyPositiveCut}.

\begin{algorithm}[!ht]
	\begin{algorithmic}[1]
		\Input $C \in \mathbb{S}^d$ for which we want to determine if $C \in \mathcal{CP}^d$ or not.
		\Function{CompletelyPositiveCut}{$C$}
			\State $c \gets \mat^*(C)$
			\State $\conr \gets 1$
			\State $\Call{Oracle}{x} \gets \Call{TestCopositive}{\mat(x)}$
			\State \Return $\mat(\Call{ACCP}{c, \textsc{Oracle}, \conr})$
		\EndFunction
	\end{algorithmic}
	\caption{A wrapper function to determine if a matrix is completely positive by solving \eqref{eq:CompletelyPositiveProblem}}
	\label{alg:CompletelyPositiveCut}
\end{algorithm}

\subsection{A Note on Complexity}
\label{subsec:Complexity}
Our aim in this paper is to propose an algorithm with good practical performance. This is why we placed emphasis on a robust copositivity check, constraint pruning, and efficient computation of the analytic center. However, such an algorithm does not lend itself well to a formal complexity analysis. For instance, to the best of the authors' knowledge, the only analysis in the literature of an analytic center cutting plane method with constraint pruning is due to Atkinson and Vaidya \cite{atkinson1995cutting}. Although the number of constraints in their algorithm is technically bounded by a polynomial of $n$, this bound is so large as to be uninteresting in practice.

The analysis that perhaps comes closest to covering our algorithm is the survey by Goffin and Vial \cite{goffin2002convex}, who find a polynomial number of iterations for an analytic center cutting plane method with deep cuts. Their method only uses linear constraints, and does not prune cuts. Moreover, the method of recovering a feasible solution after adding a deep cut is different from the infeasible start Newton method we use.

Nevertheless, we compared our method numerically to Goffin and Vial's, and found that our method exhibits somewhat better numerical performance on our test set. In particular, Goffin and Vial's method struggles earlier to approximate the analytic center. Whereas we could solve the problems in our test set up to a relative gap of $10^{-6}$, Goffin and Vial's method sometimes failed to recover a point in the feasible set when the relative gap was still of the order $10^{-5}$. The condition number of their linear systems had become very large at this point, explaining the inaccuracy. At this level of the relative gap, the condition number of the system \eqref{eq:NewtonSystemDeltaX} in our algorithm was somewhat lower.

In short, while our method is not covered by a formal complexity analysis, we do prefer it over other algorithms in the literature for numerical reasons.

\section{Numerical Experiments}
\label{sec:NumericalExperiments}

\subsection{Extremal Matrices of the $6 \times 6$ Doubly Nonnegative Cone}
\label{subsec:ExtremalDNNMatrices}
We test Algorithm \ref{alg:CompletelyPositiveCut} on extremal matrices from the doubly nonnegative cone. Ten of such $6 \times 6$ matrices were proposed in \cite[Appendix B]{badenbroek2019simulated}. We run Algorithm \ref{alg:CompletelyPositiveCut} on these matrices, and record the number of calls to $\Call{TestCopositive}{}$. For the sake of comparison, we also applied the ellipsoid method of Yudin and Nemirovski \cite{yudin1976informational}. The termination criterion for the Ellipsoid method is similar to that in Algorithm \ref{alg:ACCP}, i.e. the relative gap can be at most $10^{-6}$. The only difference is that in the case of the Ellipsoid method, the lower bound is computed through minimization over the current ellipsoid, not over some outer approximation $\mathcal{Q}$. 

The results are shown in Table \ref{tab:DNNBoundary}. We record the final objective value for all instances and both methods, as well as the number of calls to $\Call{TestCopositive}{}$. The reason to report this number of calls is that the oracle performs the theoretically intractable part of these methods: testing if a matrix is copositive. All other parts of the ellipsoid method or Algorithm \ref{alg:ACCP} complete in polynomial time for each oracle call. Hence, to get the best performance for larger matrices, one would like to minimize the number of oracle calls.

\begin{table}[h]
	\centering
	\begin{tabular}{lllll}
		\toprule
		& \multicolumn{2}{c}{Final objective value} & \multicolumn{2}{c}{$\Call{TestCopositive}{}$ calls}\\
		\cmidrule(lr){2-3} \cmidrule(lr){4-5}
		Name & Algorithm \ref{alg:ACCP} & Ellipsoid method & Algorithm \ref{alg:ACCP} & Ellipsoid method\\
  		\midrule
		\texttt{extremal\_rand\_1} & -0.28140 & -0.28139 & 169 & 8560\\
		\texttt{extremal\_rand\_2} & -0.72121 & -0.72123 & 166 & 8030\\
		\texttt{extremal\_rand\_3} & -0.73676 & -0.73676 & 163 & 8598\\
		\texttt{extremal\_rand\_4} & -0.54867 & -0.54867 & 164 & 7910\\
		\texttt{extremal\_rand\_5} & -0.92462 & -0.92460 & 177 & 8546\\
		\texttt{extremal\_rand\_6} & -1.42946 & -1.42946 & 168 & 8184\\
		\texttt{extremal\_rand\_7} & -1.67891 & -1.67889 & 168 & 9119\\
		\texttt{extremal\_rand\_8} & -1.24450 & -1.24450 & 165 & 8126\\
		\texttt{extremal\_rand\_9} & -1.04975 & -1.04974 & 176 & 8318\\
		\texttt{extremal\_rand\_10} & -0.68582 & -0.68583 & 167 & 7950\\
		\bottomrule
	\end{tabular}
	\caption{Objective values returned by Algorithm \ref{alg:CompletelyPositiveCut} and by the Ellipsoid method, applied to the matrices from \cite[Appendix B]{badenbroek2019simulated}.}
	\label{tab:DNNBoundary}
\end{table}

As can be seen from Table \ref{tab:DNNBoundary}, both methods manage to find deep cuts that separate the matrices from the completely positive cone. However, Algorithm \ref{alg:ACCP} does this with roughly 50 times fewer calls to the copositivity oracle.

\subsection{Matrices on the Boundary of the Doubly Nonnegative Cone in Higher Dimensions}
To investigate how the algorithm scales, we also generated test instances in higher dimensions. To the best of our knowledge, a complete characterization of the extremal rays of the $d \times d$ doubly nonnegative cone is unknown for $d > 6$. (See the corollary to Theorem 3.1, and Propositions 5.1 and 6.1 in Ycart \cite{ycart1982extremales} for the extremal matrices for $d \leq 6$.) Hence, we use a semidefinite programming heuristic to find doubly nonnegative matrices in these dimensions which are not completely positive.

The matrices used in Section \ref{subsec:ExtremalDNNMatrices} are $6 \times 6$ doubly nonnegative matrices $C$ with rank $3$ and the entries $C_{i,i+1} = 0$ for all $i \in \{1, ..., 5\}$. 
This pattern of zeros can of course be extended to higher dimensions, but the low rank criterion is not tractable in semidefinite programming. The standard trick to find a low-rank solution -- which we also adopt -- is to minimize the trace of the matrix variable, see e.g. Fazel, Hindi, and Boyd \cite{fazel2001rank} and the references therein. To create a $d \times d$ test instance, we thus run the procedure in Algorithm \ref{alg:GenerateTest}.

\begin{algorithm}[!ht]
	\begin{algorithmic}[1]
		\Input Dimension $d$ of a random matrix $C \in \mathbb{S}^d$ to generate.
		\State $R_0 \in \mathbb{R}^{d \times d}$ is a matrix whose elements are samples from a standard normal distribution
		\State $R \gets |R_0| + |R_0|^\top$, where $|R_0| = [|(R_0)_{ij}|]$ is the element-wise absolute value
		\State \label{line:SDP} Let $C^*$ be an (approximately) optimal solution to
		\begin{align*}
			\inf_C & \tr C + \tfrac{1}{2} d \| C - R \|\\
			\text{subject to } & C_{i,i+1} = 0 && \forall i \in \{1, ..., d-1\}\\
			& C \succeq 0, C \geq 0.
		\end{align*}
		\For{$j \in \{1, ..., 10\}$} \label{line:BeginFor}
			\State Set all eigenvalues of $C^*$ smaller than $10^{-6}$ to zero
			\State Set all elements of $C^*$ smaller than $10^{-4}$ to zero
		\EndFor \label{line:EndFor}
		\State \Return $C^* / \| C^* \|$
	\end{algorithmic}
	\caption{A heuristic procedure to generate random matrices on the boundary of the doubly nonnegative cone}
	\label{alg:GenerateTest}
\end{algorithm}

The objective in Line \ref{line:SDP} of Algorithm \ref{alg:GenerateTest} includes two terms: the term $\tr C$ to get a low-rank solution, and the term $\tfrac{1}{2} d \| C - R \|$ to get a solution close to our random matrix $R$. Without this last term, the optimal solution of the problem would be the zero matrix. The weight $\frac{1}{2} d$ was chosen because numerical experiments suggested this weight leads to solutions with low rank, but not rank zero, for the dimensions in our test set.
The solution $C^*$ computed in Line \ref{line:SDP} by interior point methods still lies in the interior of the doubly nonnegative cone. To project this solution to the boundary of the doubly nonnegative cone, we run the Lines \ref{line:BeginFor} to \ref{line:EndFor}.

For each $d \in \{6,7,8,9,10,15,20,25\}$, we generated ten test instances with Algorithm \ref{alg:GenerateTest}. Such an instance $C$ is only included in the final test set if Algorithm \ref{alg:CompletelyPositiveCut} returns an $X$ such that $\langle C, X \rangle < -0.01$, which was almost always the case. In those few cases where $\langle C, X \rangle \geq -0.01$, a new instance was generated. Hence, we end up with ten $d \times d$ doubly nonnegative matrices that are not completely positive, for each $d \in \{6,7,8,9,10,15,20,25\}$. These instances are available at \url{https://github.com/rileybadenbroek/CopositiveAnalyticCenter.jl/tree/master/test}.

Algorithm \ref{alg:CompletelyPositiveCut} is applied to each of these instances, and the total number of calls to $\Call{TestCopositive}{}$ is reported in Figure \ref{fig:OracleCalls}. (We do not report these results as in Table \ref{tab:DNNBoundary} since there are 80 instances, and running the ellipsoid method for all of them would take too much time.) As one can see, the number of oracle calls for one of our test instances with dimension $d$ is roughly $7 d^{5/3}$.

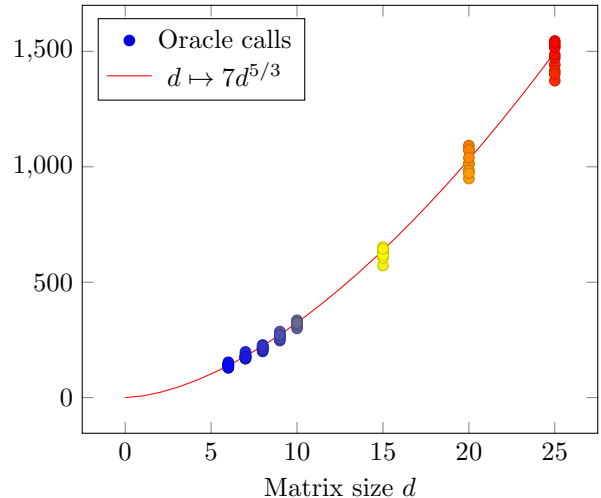
\begin{figure}
	\centering
	\begin{tikzpicture}
		\begin{axis}[xlabel={Matrix size $d$}, xmin=0, enlargelimits, ymin = 0, legend entries={Oracle calls, $d \mapsto 7 d^{5/3}$}, legend pos= north west]
			\addplot+[scatter,only marks,point meta=y] table[x=n, y=calls]
			{calls_accp_merged.txt};
			\addplot+[no marks,domain=0:25] {7*x^(5/3)};
		\end{axis}
	\end{tikzpicture}
	\caption{Number of oracle calls in Algorithm \ref{alg:CompletelyPositiveCut} for the $d \times d$ test instances generated with Algorithm \ref{alg:GenerateTest}}
	\label{fig:OracleCalls}
\end{figure}

\section{Conclusion}
We have proposed an analytic center cutting plane algorithm to separate a matrix from the completely positive cone. This algorithm solves an optimization problem over the copositive cone, where membership of the copositive cone is tested through a mixed-integer linear program.
We have emphasized stable numerical performance, which leads to an algorithm for which we do not have a formal complexity analysis. On the other hand, the numerical results are encouraging. 
In particular, the number of oracle calls to test matrix copositivity grows roughly like $O(d^2)$ for $d\times d$ matrices. Thus one can leverage the recent progress on testing matrix copositivity \cite{badenbroek2019simulated} or the similar method by Anstreicher \cite{anstreicher2020testing}. We have therefore made some computational progress on an open problem formulated by Berman, D\"ur, and Shaked-Monderer \cite{berman2015open}.
It is worthwhile to note that our algorithm can be applied to any copositive optimization problem, as long as an upper bound on the norm of the optimal solution is known. The code is available at \url{https://github.com/rileybadenbroek/CopositiveAnalyticCenter.jl}.

%\bibliographystyle{abbrv}
%\bibliography{Bibliography}

\begin{thebibliography}{10}

\bibitem{anstreicher2020testing}
K.~M. Anstreicher.
\newblock Testing copositivity via mixed--integer linear programming.
\newblock {\em Optimization Online preprint}, 2020.
\newblock \url{http://www.optimization-online.org/DB_HTML/2020/03/7659.html}.

\bibitem{atkinson1995cutting}
D.~S. Atkinson and P.~M. Vaidya.
\newblock A cutting plane algorithm for convex programming that uses analytic
  centers.
\newblock {\em Mathematical Programming}, 69(1-3):1--43, 1995.

\bibitem{badenbroek2019simulated}
R.~Badenbroek and E.~de~Klerk.
\newblock Simulated annealing with hit-and-run for convex optimization:
  rigorous complexity analysis and practical perspectives for copositive
  programming.
\newblock {\em arXiv preprint arXiv:1907.02368}, 2019.

\bibitem{berman2015open}
A.~Berman, M.~Dur, and N.~Shaked-Monderer.
\newblock Open problems in the theory of completely positive and copositive
  matrices.
\newblock {\em Electronic Journal of Linear Algebra}, 29(1):46--58, 2015.

\bibitem{bomze2018building}
I.~M. Bomze.
\newblock Building a completely positive factorization.
\newblock {\em Central European Journal of Operations Research},
  26(2):287--305, 2018.

\bibitem{bomze2002solving}
I.~M. Bomze and E.~De~Klerk.
\newblock Solving standard quadratic optimization problems via linear,
  semidefinite and copositive programming.
\newblock {\em Journal of Global Optimization}, 24(2):163--185, 2002.

\bibitem{bomze2011quadratic}
I.~M. Bomze, F.~Jarre, and F.~Rendl.
\newblock Quadratic factorization heuristics for copositive programming.
\newblock {\em Mathematical Programming Computation}, 3(1):37--57, 2011.

\bibitem{boyd2011accm}
S.~Boyd, L.~Vandenberghe, and J.~Skaf.
\newblock Lecture notes analytic center cutting-plane method.
\newblock
  \url{https://pdfs.semanticscholar.org/368e/4045bf62a2bb3fb4a39f0b0b86d7ab295964.pdf},
  April 2011.

\bibitem{bundfuss2008algorithmic}
S.~Bundfuss and M.~D{\"u}r.
\newblock Algorithmic copositivity detection by simplicial partition.
\newblock {\em Linear Algebra and its Applications}, 428(7):1511--1523, 2008.

\bibitem{bundfuss2009adaptive}
S.~Bundfuss and M.~D{\"u}r.
\newblock An adaptive linear approximation algorithm for copositive programs.
\newblock {\em SIAM Journal on Optimization}, 20(1):30--53, 2009.

\bibitem{burer2009copositive}
S.~Burer.
\newblock On the copositive representation of binary and continuous nonconvex
  quadratic programs.
\newblock {\em Mathematical Programming}, 120(2):479--495, 2009.

\bibitem{burer2013separation}
S.~Burer and H.~Dong.
\newblock Separation and relaxation for cones of quadratic forms.
\newblock {\em Mathematical Programming}, 137(1-2):343--370, 2013.

\bibitem{deklerk2002approximation}
E.~De~Klerk and D.~V. Pasechnik.
\newblock Approximation of the stability number of a graph via copositive
  programming.
\newblock {\em SIAM Journal on Optimization}, 12(4):875--892, 2002.

\bibitem{dickinson2019new}
P.~J. Dickinson.
\newblock A new certificate for copositivity.
\newblock {\em Linear Algebra and its Applications}, 569:15--37, 2019.

\bibitem{dickinson2012linear}
P.~J. Dickinson and M.~D\"ur.
\newblock Linear-time complete positivity detection and decomposition of sparse
  matrices.
\newblock {\em SIAM Journal on Matrix Analysis and Applications},
  33(3):701--720, 2012.

\bibitem{dickinson2014computational}
P.~J. Dickinson and L.~Gijben.
\newblock On the computational complexity of membership problems for the
  completely positive cone and its dual.
\newblock {\em Computational Optimization and Applications}, 57(2):403--415,
  2014.

\bibitem{elser2017matrix}
V.~Elser.
\newblock Matrix product constraints by projection methods.
\newblock {\em Journal of Global Optimization}, 68(2):329--355, 2017.

\bibitem{fazel2001rank}
M.~Fazel, H.~Hindi, and S.~P. Boyd.
\newblock A rank minimization heuristic with application to minimum order
  system approximation.
\newblock In {\em Proceedings of the 2001 American Control Conference},
  volume~6, pages 4734--4739. IEEE, 2001.

\bibitem{gaddum1958linear}
J.~W. Gaddum et~al.
\newblock Linear inequalities and quadratic forms.
\newblock {\em Pacific Journal of Mathematics}, 8(3):411--414, 1958.

\bibitem{goffin1996complexity}
J.-L. Goffin, Z.-Q. Luo, and Y.~Ye.
\newblock Complexity analysis of an interior cutting plane method for convex
  feasibility problems.
\newblock {\em SIAM Journal on Optimization}, 6(3):638--652, 1996.

\bibitem{goffin1993computation}
J.-L. Goffin and J.-P. Vial.
\newblock On the computation of weighted analytic centers and dual ellipsoids
  with the projective algorithm.
\newblock {\em Mathematical Programming}, 60(1-3):81--92, 1993.

\bibitem{goffin2002convex}
J.-L. Goffin and J.-P. Vial.
\newblock Convex nondifferentiable optimization: A survey focused on the
  analytic center cutting plane method.
\newblock {\em Optimization Methods and Software}, 17(5):805--867, 2002.

\bibitem{groetzner2018factorization}
P.~Groetzner and M.~D{\"u}r.
\newblock A factorization method for completely positive matrices.
\newblock {\em Optimization Online}, 2018.

\bibitem{hiriart2010variational}
J.-B. Hiriart-Urruty and A.~Seeger.
\newblock A variational approach to copositive matrices.
\newblock {\em SIAM Review}, 52(4):593--629, 2010.

\bibitem{jarre2009computation}
F.~Jarre and K.~Schmallowsky.
\newblock On the computation of ${C}^*$ certificates.
\newblock {\em Journal of Global Optimization}, 45(2):281, 2009.

\bibitem{kemperman1992covariance}
J.~Kemperman and M.~Skibinsky.
\newblock Covariance spaces for measures on polyhedral sets.
\newblock {\em Lecture Notes-Monograph Series}, pages 182--195, 1992.

\bibitem{kong2013scheduling}
Q.~Kong, C.-Y. Lee, C.-P. Teo, and Z.~Zheng.
\newblock Scheduling arrivals to a stochastic service delivery system using
  copositive cones.
\newblock {\em Operations Research}, 61(3):711--726, 2013.

\bibitem{lasserre2014new}
J.~B. Lasserre.
\newblock New approximations for the cone of copositive matrices and its dual.
\newblock {\em Mathematical Programming}, 144(1-2):265--276, 2014.

\bibitem{li2014distributionally}
X.~Li, K.~Natarajan, C.-P. Teo, and Z.~Zheng.
\newblock Distributionally robust mixed integer linear programs: Persistency
  models with applications.
\newblock {\em European Journal of Operational Research}, 233(3):459--473,
  2014.

\bibitem{motzkin1965maxima}
T.~S. Motzkin and E.~G. Straus.
\newblock Maxima for graphs and a new proof of a theorem of tur{\'a}n.
\newblock {\em Canadian Journal of Mathematics}, 17:533--540, 1965.

\bibitem{murty1987some}
K.~G. Murty and S.~N. Kabadi.
\newblock Some {NP}-complete problems in quadratic and nonlinear programming.
\newblock {\em Mathematical Programming}, 39(2):117--129, 1987.

\bibitem{natarajan2011mixed}
K.~Natarajan, C.~P. Teo, and Z.~Zheng.
\newblock Mixed 0-1 linear programs under objective uncertainty: A completely
  positive representation.
\newblock {\em Operations Research}, 59(3):713--728, 2011.

\bibitem{nie2014Atruncated}
J.~Nie.
\newblock The $\mathcal{A}$-truncated ${K}$-moment problem.
\newblock {\em Foundations of Computational Mathematics}, 14(6):1243--1276,
  2014.

\bibitem{nie2018complete}
J.~Nie, Z.~Yang, and X.~Zhang.
\newblock A complete semidefinite algorithm for detecting copositive matrices
  and tensors.
\newblock {\em SIAM Journal on Optimization}, 28(4):2902--2921, 2018.

\bibitem{parrilo2000structured}
P.~A. Parrilo.
\newblock {\em Structured semidefinite programs and semialgebraic geometry
  methods in robustness and optimization}.
\newblock PhD thesis, California Institute of Technology, 2000.

\bibitem{pena2007computing}
J.~Pe{\~n}a, J.~Vera, and L.~F. Zuluaga.
\newblock Computing the stability number of a graph via linear and semidefinite
  programming.
\newblock {\em SIAM Journal on Optimization}, 18(1):87--105, 2007.

\bibitem{renegar2001mathematical}
J.~Renegar.
\newblock {\em A Mathematical View of Interior-Point Methods in Convex
  Optimization}.
\newblock SIAM, 2001.

\bibitem{sikiric2020simplex}
M.~D. Sikiri{\'c}, A.~Sch{\"u}rmann, and F.~Vallentin.
\newblock A simplex algorithm for rational cp-factorization.
\newblock {\em Mathematical Programming}, 2020.
\newblock \url{https://doi.org/10.1007/s10107-020-01467-4}.

\bibitem{sponsel2014factorization}
J.~Sponsel and M.~D{\"u}r.
\newblock Factorization and cutting planes for completely positive matrices by
  copositive projection.
\newblock {\em Mathematical Programming}, 143(1-2):211--229, 2014.

\bibitem{xia2018globally}
W.~Xia, J.~C. Vera, and L.~F. Zuluaga.
\newblock Globally solving non-convex quadratic programs via linear integer
  programming techniques.
\newblock {\em INFORMS Journal on Computing}, 2018.

\bibitem{ycart1982extremales}
B.~Ycart.
\newblock Extr{\'e}males du c{\^o}ne des matrices de type non n{\'e}gatif,
  {\`a} coefficients positifs ou nuls.
\newblock {\em Linear Algebra and its Applications}, 48:317--330, 1982.

\bibitem{yildirim2012accuracy}
E.~A. Y{\i}ld{\i}r{\i}m.
\newblock On the accuracy of uniform polyhedral approximations of the
  copositive cone.
\newblock {\em Optimization Methods and Software}, 27(1):155--173, 2012.

\bibitem{yudin1976informational}
D.~Yudin and A.~Nemirovski.
\newblock Informational complexity and effective methods of solution of convex
  extremal problems.
\newblock {\em Economics and Mathematical Methods}, 12:357--369, 1976.

\end{thebibliography}

\end{document}